\documentclass[12pt,a4paper,reqno]{amsart}
\usepackage{hyperref}
\usepackage{cleveref}
\usepackage{amssymb}
\usepackage{verbatim}
\usepackage{amscd}
\usepackage{enumerate}
\usepackage{graphicx}
\usepackage{siunitx}
\usepackage{tikz-cd}
\usepackage{stix}
\usepackage{bm}
\usepackage{mathtools}
\usepackage[tableposition=top]{caption}
\usepackage{booktabs,dcolumn}
\usepackage{tikz}
\usetikzlibrary{shapes.geometric, arrows}

\numberwithin{equation}{section}

\theoremstyle{plain}

\newtheorem{theorem}{Theorem}[section]

\newtheorem{lemma}[theorem]{Lemma}

\theoremstyle{definition}

\newtheorem{remark}[theorem]{Remark}

\newcommand\R{\mathbb{R}}

\begin{document}

\title{A Maclaurin type inequality}

\author{Terence Tao}
\address{UCLA Department of Mathematics, Los Angeles, CA 90095-1555.}
\email{tao@math.ucla.edu}


\begin{abstract}  The classical Maclaurin inequality asserts that the elementary symmetric means
$$ s_k(y) \coloneqq \frac{1}{\binom{n}{k}} \sum_{1 \leq i_1 < \dots < i_k \leq n} y_{i_1} \dots y_{i_k}$$
obey the inequality $s_\ell(y)^{1/\ell} \leq s_k(y)^{1/k}$ whenever $1 \leq k \leq \ell \leq n$ and $y = (y_1,\dots,y_n)$ consists of non-negative reals. We establish a variant
$$ |s_\ell(y)|^{\frac{1}{\ell}} \ll \frac{\ell^{1/2}}{k^{1/2}} \max (|s_k(y)|^{\frac{1}{k}}, |s_{k+1}(y)|^{\frac{1}{k+1}})$$
of this inequality in which the $y_i$ are permitted to be negative.  In this regime the inequality is sharp up to constants.  Such an inequality was previously known without the $k^{1/2}$ factor in the denominator.
\end{abstract}

\maketitle


\section{Introduction}

Given $n$ real numbers $y = (y_1,\dots,y_n) \in \R^n$ and $0 \leq k \leq n$, let $s_k(y)$ denote the elementary symmetric means
$$ s_k(y) \coloneqq \frac{1}{\binom{n}{k}} \sum_{1 \leq i_1 < \dots < i_k \leq n} y_{i_1} \dots y_{i_k}$$
(thus for instance $s_0(y)=1$).  These means arise as normalized coefficients of the degree $n$ polynomial with roots $y_1,\dots,y_n$:
\begin{equation}\label{poly}
 \prod_{i=1}^n (z-y_i) = \sum_{k=0}^n (-1)^k \binom{n}{k} s_k(y) z^{n-k}.
\end{equation}
We call a $n+1$-tuple $(s_0,\dots,s_n)$ of real numbers \emph{attainable} if it is of the form $(s_0(y),\dots,s_n(y))$ for some $y$, or equivalently if the polynomial $\sum_{k=0}^n (-1)^k \binom{n}{k} s_k z^{n-k}$ is monic with all roots real.  For example, the attainable $3$-tuples take the form
$$ (s_0,s_1,s_2) = \left(1, \frac{y_1+y_2}{2}, y_1 y_2\right)$$
and range in the set $\{ (s_0,s_1,s_2): s_0=1, s_2 \leq s_1^2 \}$.  More generally, for any $n$ the set of attainable $n+1$-tuples is some closed semi-algebraic set \cite{sylvester}; see \cite[\S 3]{niculescu} for an explicit description in the cases $n=3,4$, which is related to the discriminant of a degree $n$ polynomial.  As another example, by taking all the $y_i$ to be $1$, we see that the $n+1$-tuple $(1,\dots,1)$ is attainable for any $n$.

There are several inequalities known for attainable tuples $(s_0,\dots,s_n)$, most famously the classical \emph{Newton inequality} \cite{newton}
\begin{equation}\label{newton}
 s_{k-1} s_{k+1} \leq s_k^2
\end{equation}
for any $1 \leq k < n$.  If the underyling real variables $y_1,\dots,y_n$ are non-negative, the Newton inequalities can be used to easily establish the well-known \emph{Maclaurin inequalities} \cite[pp. 80-81]{maclaurin}
\begin{equation}\label{maclaurin}
 s_1 \geq s_2^{\frac{1}{2}} \geq \dots \geq s_n^{\frac{1}{n}}.
\end{equation}
Many further inequalities in the non-negative case are known; a classical reference in this regard is \cite{polya}, and some more recently discovered inequalities may be found for instance in \cite{mitrinovic}, \cite{sato}, \cite{mitev}, \cite{mc}, \cite{fw}, \cite{cgs}, \cite{tin}.  We also note the inequalities of Lin and Trudinger \cite{lint} which apply when a certain specified number of the $y_i$ (or the $s_k$) are known to be positive, and which have numerous applications to the study of partial differential equations.

Returning to the case of arbitrary real $y_1,\dots,y_n$, the Newton inequalities were strengthened by Rosset \cite{rosset} to the cubic Newton inequalities
$$ 4 (s_{k+2}^2 - s_{k+1} s_{k+3}) ( s_{k+1}^2 - s_k s_{k+2} ) \geq (s_{k+1} s_{k+2} - s_k s_{k+3})^2,$$
valid for $0 \leq k \leq n-3$ and arbitrary attainable $(s_0,\dots,s_n)$.  A further quartic strengthening of this inequality was then established in \cite{niculescu}; these inequalities are related to the solution by radicals of the general cubic and quartic respectively. Some further inequalities between the $s_k$ may be found in (or derived from) \cite{hunter}, \cite{timofte, timofte-2, timofte-3}, \cite{rukhin}, \cite{nicu-2}.  On the other hand, the Maclaurin inequality \eqref{maclaurin} can fail in general once the $y_i$ are not required to be non-negative.  Indeed, it is possible for an intermediate entry $s_k$ in an attainable tuple $(s_0,\dots,s_n)$ to vanish without the subsequent entries $s_{k+1},\dots,s_n$ vanishing.  A key example occurs when $n$ is even, and the $y_1,\dots,y_n$ take values in $\{-1,+1\}$, with exactly $n/2$ of the $y_i$ equal to each of the signs $+1, -1$.  From the binomial expansion
$$ (x-1)^{n/2} (x+1)^{n/2} = (x^2-1)^{n/2} = \sum_{j=0}^{n/2} (-1)^j \binom{n/2}{j} x^{2j}$$
and \eqref{poly} we see that
\begin{equation}\label{sk-form}
s_k = s_k(y_1,\dots,y_n) = (-1)^{k/2} \frac{\binom{n/2}{k/2}}{\binom{n}{k}} 1_{k \text{ even}}.
\end{equation}
Using the standard inequalities
\begin{equation}\label{standard}
    \frac{n^k}{k^k} \leq \binom{n}{k} \leq \frac{n^k}{k!}
\end{equation}
for all $1 \leq k \leq n$, which follow immediately from rearranging the obvious inequalities
   $$ \prod_{j=0}^{k-1} (1 - \frac{j}{k}) \leq \prod_{j=0}^{k-1} (1 - \frac{j}{n}) \leq 1,$$
combined with the weak Stirling bounds
\begin{equation}\label{weak-stirling}
 \frac{k^k}{e^k} \leq k! \leq k^k
\end{equation}
(where the first inequality is immediate from the Taylor expansion $e^k = \sum_{j=0}^k \frac{k^j}{j!} \geq \frac{k^k}{k!} $ of $e^k$) we thus see for this specific example that
\begin{equation}\label{k-even}
   |s_k|^{\frac{1}{k}} \asymp \frac{k^{1/2}}{n^{1/2}} 1_{k \text{ even}}
\end{equation}
for $1 \leq k \leq n$. Here we use the standard asymptotic notation $X \ll Y$, $Y \gg X$ or $X=O(Y)$ to denote the estimate $|X| \leq C Y$ for some absolute constant $C>0$, and $X \asymp Y$ to denote the bound $X \ll Y \ll X$.
Thus we have a breakdown of \eqref{maclaurin} in the example \eqref{sk-form}, even when absolute value signs are placed on the $s_k$: smallness (or even vanishing) of one quantity $|s_k|^{\frac{1}{k}}$ does not imply smallness of all subsequent quantities $|s_\ell|^{\frac{1}{\ell}}$ for $k < \ell \leq n$.

However, it was observed in\footnote{This assertion is not explicit in \cite{gy}, which focused more on the related question of whether smallness of $|s_{k+1}|/|s_k|$ and $|s_{k+2}|/|s_k|$ implied smallness of $|s_\ell|/|s_k|$ for all larger $\ell$, but as observed by several authors \cite{dhh,mrt}, the methods in \cite{gy} extend to cover the situation discussed here.} \cite{gy} that if \emph{two} consecutive quantities $|s_k|^{\frac{1}{k}}, |s_{k+1}|^{\frac{1}{k+1}}$ were small in magnitude, then the subsequent entries $|s_\ell|^{\frac{1}{\ell}}$ for $k+1 < \ell \leq n$ were also small in magnitude, though with some loss in the bounds.  In fact, as observed in \cite{dhh,mrt}, the arguments in \cite{gy} yield an inequality of the form
\begin{equation}\label{prev-bound}
 |s_\ell|^{\frac{1}{\ell}} \ll \ell^{1/2} \max_{k' = k, k+1} |s_{k'}|^{\frac{1}{k'}}
\end{equation}
whenever $1 \leq k \leq \ell \leq n$.  For the convenience of the reader, we give a proof of \eqref{prev-bound} in Section \ref{prelim-sec}.

A comparison of \eqref{prev-bound} with \eqref{k-even} suggests that the former inequality is not sharp; also, when $\ell = k, k+1$, the factor $\ell^{1/2}$ in \eqref{prev-bound} is clearly unnecessary.
The main result of this note is the following improvement of \eqref{prev-bound}, which was conjectured (in an equivalent form) by Chin Ho Lee\footnote{See \url{https://mathoverflow.net/questions/446254}.}:

\begin{theorem}[Maclaurin type inequality]\label{thm-main}  Let $(s_0,\dots,s_n)$ be an attainable tuple.  Then we have
\begin{equation}\label{main-eq}
        |s_\ell|^{\frac{1}{\ell}} \ll \max_{k' = k, k+1} \left( \frac{\ell}{k'}\right)^{1/2} |s_{k'}|^{\frac{1}{k'}}
        \asymp  \left( \frac{\ell}{k}\right)^{1/2} \max_{k' = k, k+1} |s_{k'}|^{\frac{1}{k'}}
\end{equation}
for all $1 \leq k < \ell \leq n$.
\end{theorem}

This bound is optimal up to constants, as can be seen from \eqref{k-even}. In particular, it may be used to make minor improvements to some of the error bounds in \cite{dhh,mrt}.

\begin{remark}\label{eqv}   The bound \eqref{main-eq} can of course be expressed in terms of the non-normalized elementary symmetric functions
$$ S_k(y) \coloneqq \binom{n}{k} s_k(y) = \sum_{1 \leq i_1 < \dots < i_k \leq n} y_{i_1} \dots y_{i_k}.$$
Indeed, from \eqref{standard}, we see that \eqref{main-eq} is equivalent to the inequality
$$ |S_\ell(y)|^{\frac{1}{\ell}} \ll  \left( \frac{k}{\ell}\right)^{1/2}\max_{k' = k, k+1} |S_{k'}(y)|^{\frac{1}{k'}}$$
for all real $y_1,\dots,y_n$ and all $1 \leq k < \ell \leq n$.  Using the weak Stirling bounds \eqref{weak-stirling}, we may also write the bound \eqref{main-eq} in the following alternate form: if one has the bounds
$$ |S_k(y)|^2 \leq \frac{\theta^k}{k!}$$
and
$$ |S_{k+1}(y)|^2 \leq \frac{\theta^{k+1}}{(k+1)!}$$
for some $1 \leq k < n$ and $\theta>0$, then one has
$$ |S_\ell(y)|^2 \leq \frac{(C\theta)^\ell}{\ell!}$$
for all $k \leq \ell \leq n$ and some absolute constant $C$.  This is the form of the inequality that was originally conjectured by Lee.
\end{remark}

The approach in \cite{gy} to proving bounds such as \eqref{prev-bound} was based on the arithmetic-geometric mean inequality, which among other things implies that
\begin{equation}\label{amgm}
|s_n(y)|^{\frac{2}{n}} = (y_1^2 \dots y_n^2)^{\frac{1}{n}} \leq \frac{1}{n} (y_1^2 + \dots + y_n^2)
\end{equation}
for any real numbers $y_1,\dots,y_n$ (which are not required to be non-negative); this readily implies the $k=1, \ell=n$ case of \eqref{prev-bound}, and the general case can be deduced by using some well-known symmetry properties of the space of attainable tuples.  This is analogous to the standard proof of the Newton inequality \eqref{newton} that starts from the special case $k=1, n=2$ and then establishes the general case using the aforementioned symmetries. Our approach will proceed instead by exploiting the following inequality which, while simple to prove, has not appeared previously in the literature to the author's knowledge.

\begin{theorem}[New inequality]\label{thm-new}  Let $(s_0,\dots,s_n)$ be an attainable tuple.  Then for any $r>0$ and $1 \leq \ell \leq n$, one has
\begin{equation}\label{r1}
 \sum_{m=0}^\ell \binom{\ell}{m} |s_m| r^m \geq (1+ |s_\ell|^{2/\ell} r^2)^{\ell/2}
\end{equation}
or equivalently
\begin{equation}\label{r2}
 \sum_{m=0}^\ell \binom{\ell}{m} |s_m| r^{\ell-m} \geq (|s_\ell|^{2/\ell} +r^2)^{\ell/2}.
\end{equation}
\end{theorem}

We establish this inequality in Section \ref{new-ineq}.  Observe from \eqref{sk-form} and the binomial theorem that these inequalities are completely sharp (at least when $\ell=n$), thus giving support to the heuristic that the tuple \eqref{sk-form} has extremal behavior amongst the set of all attainable tuples.   Morally speaking, the estimates in Theorem \ref{thm-new} (working for instance in the model case $\ell=n$ and $|s_n|=1$, which is easy to reduce to) should then follow by applying \eqref{r1} with\footnote{This choice of $r$ is suggested by the law of large numbers (or the de Moivre--Laplace theorem), which among other things indicates that the binomial distribution $m \mapsto (1+r^2)^{-n/2} \binom{n/2}{m/2} 1_{m \text{ even}} r^m$ peaks at $m \approx \frac{r^2}{1+r^2} n$.} $r \asymp (k/n)^{1/2}$ and using \eqref{standard}, provided that we can somehow neglect the contribution of the $m \neq k,k+1$ terms on the left-hand side of \eqref{r1}.  This turns out to be possible after a certain amount of technical preparation (relying on the aforementioned symmetries of the space of attainable tuples) to pass to an ``extremal'' attainable tuple.

\begin{remark}  By the fundamental theorem of algebra, every tuple $(s_0,\dots,s_n)$ of complex numbers with $s_0=1$ becomes attainable if the $y_1,\dots,y_n$ are permitted to be arbitrary complex numbers.  Thus there are no non-trivial relations between the $s_i$ in this case.
\end{remark}

\begin{remark}
The author has formalized the proof of \Cref{thm-main} in {\tt Lean 4}: see

\centerline{\url{https://github.com/teorth/symmetric_project}.}
As a byproduct of this formalization, it was provably shown that
$$  |s_\ell|^{\frac{1}{\ell}} \leq C \max_{k' = k, k+1} \left( \frac{\ell}{k'}\right)^{1/2} |s_{k'}|^{\frac{1}{k'}}
$$
for all $1 \leq k < \ell \leq n$ and admissible $(s_0,\dots,s_n)$ with\footnote{Strictly speaking, the formalization uses a more complicated expression for $C$, namely the maximum of five separate constants including $160 \times e^7$, but one can easily verify numerically that $160 \times e^7$ is (by far) the largest of the constants appearing in the maximum.} $C \coloneqq 160 \times e^7 \approx 175461$, although this constant is far from sharp.
\end{remark}

\subsection{Acknowledgments}

The author is supported by NSF grant DMS-1764034.  We thank Oisin McGuinness and Chin Ho Lee for comments and corrections, and Jairo Bochi for some references.

\section{Preliminary bounds}\label{prelim-sec}

In this section we record some preliminary bounds, in particular establishing the bound \eqref{prev-bound} using the method from \cite{gy}.

The collection of attainable tuples enjoys some useful (and well-known) symmetries, which we shall rely on frequently in our arguments:

\begin{lemma}[Symmetries of attainable tuples]\label{syms}  Let $(s_0,\dots,s_n)$ be an attainable tuple.
    \begin{itemize}
        \item[(i)] (Scaling)  For any real $\lambda$, $(s_0, \lambda s_1, \dots, \lambda^n s_n)$ is attainable.
        \item[(ii)] (Reflection)  If $s_n \neq 0$, then $(1, s_{n-1}/s_n, \dots, s_0/s_n)$ is attainable. (In particular, if $|s_n|=1$, then $\pm (s_n,\dots,s_0)$ is attainable with $\pm$ the sign of $s_n$.)
        \item[(iii)] (Truncation) If $1 \leq \ell \leq n$, then $(s_0,\dots,s_\ell)$ is attainable.
    \end{itemize}
    \end{lemma}

    \begin{proof}  We can write $s_k = s_k(y_1,\dots,y_n)$ for some real $y_1,\dots,y_n$. The claims (i), (ii) are immediate from the homogeneity identity
        $$ s_k(\lambda y_1,\dots,\lambda y_n) = \lambda^k s_k(y_1,\dots,y_n)$$
        and the reflection identity
    $$ s_k(1/y_1,\dots,1/y_n) = s_{n-k}(y_1,\dots,y_n) / s_n(y_1,\dots,y_n)$$
    respectively for all $0 \leq k \leq n$ (note that the non-vanishing of $s_n(y_1,\dots,y_n)$ implies that all the $y_1,\dots,y_n$ are non-zero).  To prove (iii), observe from $n-\ell$ applications of Rolle's theorem that the degree $\ell$ polynomial
    $$ \frac{\ell!}{n!} \frac{d^{n-\ell}}{dx^{n-\ell}} \prod_{i=1}^n (z-y_i) = \sum_{k=0}^\ell (-1)^k \binom{\ell}{k} s_k(y_1,\dots,y_n) z^{\ell-k}$$
    is monic with all roots real, and hence the tuple $(s_0(y_1,\dots,y_n),\dots,s_\ell(y_1,\dots,y_n))$ is attainable.
\end{proof}

We may now prove \eqref{prev-bound}.  By Lemma \ref{syms}(iii) we may assume without loss of generality that $\ell=n$.  The bound is trivial or vacuous when $n \leq 2$, so we may assume $n \geq 3$.

From the Newton identity
$$
\sum_{i=1}^n y_i^2 = n^2 s_1(y_1,\dots,y_n)^2 - 2 \binom{n}{2} s_2(y_1,\dots,y_n)$$
and \eqref{amgm} we obtain the bound
$$ |s_n(y_1,\dots,y_n)|^{\frac{2}{n}} \leq n s_1(y_1,\dots,y_n)^2 - (n-1) s_2(y_1,\dots,y_n).$$
By the triangle inequality, we conclude for any attainable tuple $(s_0,\dots,s_n)$ that
$$ |s_n|^{\frac{2}{n}} \leq n s_1^2 + (n-1) |s_2|\leq \max_{k=1,2} 2n |s_k|^{\frac{2}{k}}.$$
(This already establishes\footnote{When combined with Lemma \ref{syms}(iii), it also gives the inequality $|s_\ell|^{\frac{2}{\ell}} \leq \ell s_1^2 - (\ell-1) s_2$ for any $2 \leq \ell \leq n$.  See \cite[Claim 22]{bijls} for a related bound.} the $k=1$ case of \eqref{prev-bound}.) Using Lemma \ref{syms}(ii), we conclude that
$$ |s_n|^{-\frac{2}{n}} \leq \max_{k=1,2} 2n (|s_{n-k}|/|s_n|)^{\frac{2}{k}}$$
if $s_n \neq 0$.  With our assumption $n \geq 3$, this rearranges as
$$ |s_n|^{\frac{1}{n}} \leq \max_{k=1,2} (2n)^{\frac{k}{2(n-k)}} |s_{n-k}|^{\frac{1}{n-k}}.$$
We may remove the hypothesis $s_n \neq 0$, since this inequality clearly holds when $s_n$ vanishes.

Iterating this bound (and using Lemma \ref{syms}(iii)) we may establish \eqref{prev-bound} for any fixed $n$. We may therefore assume henceforth that $n \geq 6$.

It will be convenient to replace the factor $(2n)^{\frac{k}{2(n-k)}}$ by an expression which will form a telescoping product. From the monotone decreasing convex nature of $t \mapsto \frac{1}{t-1}$ for $t > 1$, one has the inequality
$$ \frac{k}{2(n-k)} \leq \int_{n-k}^n \frac{dt}{2(t-1)} = \frac{1}{2} \log \frac{n-1}{n-k-1}$$
for $k=1,2$ and\footnote{In a previous version of the paper, the condition $n \geq 4$ was omitted, as the author failed to notice that the right-hand side was infinite (and thus not usable) for $n=3$ when $k=2$.  This error was detected by {\tt Lean4} in the course of the author's attempts to formally verify the arguments in this paper.} $n \geq 4$. we conclude that
$$ |s_n|^{\frac{1}{n}} \leq \max_{k=1,2} (2n)^{\frac{1}{2} \log \frac{n-1}{n-k-1}} |s_{n-k}|^{\frac{1}{n-k}}$$
for $n \geq 4$.
A routine induction on $n$ then shows that
\begin{equation}\label{nk}
 |s_n|^{\frac{1}{n}} \leq \max_{k'=k,k+1} (2n)^{\frac{1}{2} \log \frac{n-1}{n-k'-1}}  |s_{n-k'}|^{\frac{1}{n-k'}}
\end{equation}
whenever $1 \leq k \leq n-3$.  If we again make the temporary assumption that $s_n \neq 0$, we can use Lemma \ref{syms}(ii) to conclude that
$$ |s_n|^{-\frac{1}{n}} \leq \max_{k'=k,k+1} (2n)^{\frac{1}{2} \log \frac{n-1}{n-k'-1}}  (|s_{k'}|/|s_n|)^{\frac{1}{n-k'}}$$
which rearranges as
$$ |s_n|^{\frac{1}{n}} \leq \max_{k'=k,k+1} (2n)^{\frac{n-k'}{2k'} \log \frac{n-1}{n-k'-1}}  |s_{k'}|^{\frac{1}{k'}}.$$
We may again remove the condition $s_n \neq 0$ here, since the bound clearly holds when $s_n=0$.
In the region $k \leq n/2$ (which implies $k \leq n-3$ since $n \geq 6$), we have $\log \frac{n-1}{n-k'-1} \leq \frac{k'}{n-k'} + O( \frac{k'}{n^2})$ and hence
\begin{equation}\label{sk}
    |s_n|^{\frac{1}{n}} \ll n^{1/2} \max_{k'=k,k+1} |s_{k'}|^{\frac{1}{k'}}.
\end{equation}
In the region $n/2 < k \leq n-2$, we instead use \eqref{nk} (with $k$ replaced by $n-1-k$, noting that $n-1-k \leq n-3$ in this regime since $n \geq 6$) to obtain
$$
|s_n|^{\frac{1}{n}} \leq \max_{k'=k,k+1} (2n)^{\frac{1}{2} \log \frac{n-1}{k'-1}}  |s_{k'}|^{\frac{1}{k'}}.
$$
But $\frac{n-1}{k'-1} \leq 2 + O(\frac{1}{n})$, and so in this region we also have \eqref{sk} since $\log 2 \leq 1$.  Finally, \eqref{sk} trivially holds when $k=n-1$.  We have thus established \eqref{sk} for all $1 \leq k < n$.  Since we had previously reduced to the case $\ell=n$, we obtain \eqref{prev-bound} as required.

\begin{remark} It is not difficult to make the implied constant in \eqref{prev-bound} more explicit; see \cite[Theorem 2]{dhh} or \cite[Theorem 5.2]{mrt} for examples of such explicit versions of this bound (with slightly different normalizations).  In the authors formalization of this paper in {\tt Lean4}, the bound was established with the constant
$$ C \coloneqq \max( e^{4/e} \sqrt{2}, 2 \sqrt{7} ) = 6.160089\dots,
$$
although this bound could surely be improved with additional effort.
\end{remark}

\section{Proof of Theorem \ref{thm-new}}\label{new-ineq}

We now prove Theorem \ref{thm-new}.  As in the proof of \eqref{prev-bound}, we may invoke Lemma \ref{syms}(iii) and assume without loss of generality that $\ell=n$. We may assume that $s_n \neq 0$ since the desired bound is trivial otherwise, and then by Lemma \ref{syms}(i) we may normalize $|s_n|=1$.  The bounds \eqref{r1} and \eqref{r2} are clearly equivalent (either by replacing $r$ with $1/r$, or by using Lemma \ref{syms}(ii)), so we will content ourselves with proving \eqref{r2}.  Fix an attainable tuple $(s_0,\dots,s_n)$ with $|s_n|=1$ and a real number $r>0$.  We write $s_k = s_k(y)$ for some real $y_1,\dots,y_n$; from the normalization $|s_n|=1$ we see that the $y_i$ are non-zero.  We form the polynomial
\begin{equation}\label{pz}
 P(z) \coloneqq \prod_{j=1}^n (z-y_j)
\end{equation}
of one complex variable $z$.  From \eqref{poly} and the triangle inequality we have
$$ |P(ir)| \leq \sum_{k=0}^n \binom{n}{k} |s_k| r^{n-k}$$
so it will suffice to establish the lower bound
$$ \log |P(ir)| \geq \frac{n}{2} \log(1+r^2).$$
But by taking absolute values and then logarithms in \eqref{pz}, we may expand
$$ \log |P(ir)| = \frac{1}{2} \sum_{j=1}^n \log(y_j^2 + r^2).$$
From the convexity of $t \mapsto \log(e^t+r^2)$ for $t \in \R$ and Jensen's inequality one has
$$ \frac{1}{n} \sum_{j=1}^n \log(y_j^2 + r^2) \leq \log(e^{\frac{1}{n} \log y_j^2} + r^2).$$
But from the normalization $|s_n|=1$ we have $\sum_{j=1}^n \log y_j^2=0$.  The claim follows.

\begin{remark} It is tempting to use the contour integral representation
$$ (-1)^k \binom{n}{k} s_k = \frac{1}{2\pi i} \int_\gamma \frac{P(z)}{z^{n-k+1}}\ dz$$
for any $0 \leq k \leq n$ and any contour $\gamma$ winding once anti-clockwise around the origin, together with the saddle point method, to try to control $s_k$ in terms of bounds on $P(z)$.  We were not able to do this rigorously (mostly because the geometry of the subharmonic function $\log |P(z)|$ could be moderately complicated despite the singularities being constrained to the real axis), but eventually landed upon the the more elementary approach of estimating $P$ at a single complex value $ir$ as a workable substitute for the saddle point method.  On\footnote{We thank Jairo Bochi for these references.} the other hand, the saddle point method has been successfully been used to control elementary symmetric means of random variables \cite{halasz}, \cite{major}, leading to some inequalities relating such means to additional types of mean in the case of non-negative variables \cite{bip}.
\end{remark}

\section{Main argument}

Now we can establish \Cref{thm-main}.  The idea is to work with an ``extremal'' tuple and then show that the contributions of the $m \neq k,k+1$ terms in \eqref{r2} are negligible in that case.  In order to guarantee the existence of a (near)-extremizer, it is convenient to truncate the range of the parameters $k,\ell,n$.  Let $N$ be a large natural number, and let $A = A_N$ be the best constant for which one has the inequality
\begin{equation}\label{A-def}
 |s_\ell|^{\frac{1}{\ell}} \leq A \max_{k' = k, k+1} \left( \frac{\ell}{k'}\right)^{1/2} |s_{k'}|^{\frac{1}{k'}}
\end{equation}
whenever $1 \leq k \leq \ell \leq n \leq N$ (with $k<n$) and $(s_0,\dots,s_n)$ is attainable.  By \eqref{prev-bound}, $A$ is finite (this is the only place in the argument where we exploit the upper bound of $N$).  It will suffice to show that $A$ is bounded uniformly in $N$.  Thus, we may assume for sake of contradiction that $A$ is larger than any specified constant.

By Lemma \ref{syms}(ii), (iii), we have
$$ |1/s_\ell|^{\frac{1}{\ell}} \leq A \max_{k' = k, k+1} \left( \frac{\ell}{k'}\right)^{1/2} |s_{\ell-k'}/s_\ell|^{\frac{1}{k'}}$$
whenever $1 \leq k <\ell \leq n \leq N$ and $(s_0,\dots,s_N)$ is attainable with $s_\ell \neq 0$.  This can be rearranged as
\begin{equation}\label{flip}
 |s_\ell|^{\frac{1}{\ell}} \leq \max_{k' = k, k+1} A^{\frac{k'}{\ell-k'}} \left( \frac{\ell}{k'}\right)^{\frac{k'}{2(\ell-k')}} |s_{\ell-k'}|^{\frac{1}{\ell-k'}}.
\end{equation}
Obviously this bound also holds in the $s_\ell=0$ case, so the condition $s_\ell \neq 0$ may then be removed.

The bounds \eqref{A-def}, \eqref{flip} involve positive powers of the potentially unbounded quantity $A$, which at first glance renders them unsuitable for our application.  However, we can counteract these losses by working with an ``extremal'' tuple that gains back a negative power of $A$.  Indeed, by definition of $A$, we can find $1 \leq k \leq \ell \leq n \leq N$ and an attainable tuple $(s_0,\dots,s_N)$  such that
\begin{equation}\label{near}
 |s_\ell|^{\frac{1}{\ell}} > e^{-1/N} A \max_{k' = k, k+1} \left( \frac{\ell}{k'}\right)^{1/2} |s_{k'}|^{\frac{1}{k'}}
\end{equation}
(say); one should think of this tuple as being ``near-extremal'' for the inequality \eqref{A-def}.  From \eqref{prev-bound} and the largeness of $A$, we may assume that $k$ (and hence $\ell, n, N$) is larger than any specified absolute constant; we can also assume that $k+2 \leq \ell$ since the cases $\ell = k, k+1$ are not compatible with $A$ being large. By Lemma \ref{syms}(iii) we may assume without loss of generality that $\ell=n$.  From \eqref{near} we then have $s_n$ non-zero, and by Lemma \ref{syms}(ii) we may assume that we have the normalization $|s_n|=1$.  Thus we now have
\begin{equation}\label{s-major}
 |s_{k'}|^{\frac{1}{k'}} \leq e^{1/N} A^{-1} \left( \frac{k'}{n}\right)^{1/2}
\end{equation}
for $k'=k,k+1$; note in particular the negative power of $A$ on the right-hand side. In particular this means that $k \leq n-2$, since $|s_n|=1$ and $A$ is large.  On the other hand, if we apply \eqref{flip} with $k$ replaced by $n-1-k$ and $\ell$ replaced by $n$, then
\begin{align*}
|s_n|^{\frac{1}{n}} &\leq \max_{k' = n-1-k, n-k} A^{\frac{k'}{n-k'}} \left( \frac{n}{k'}\right)^{\frac{k'}{2(n-k')}} |s_{n-k'}|^{\frac{1}{n-k'}} \\
&= \max_{k' = k, k+1} A^{\frac{n-k'}{k'}} \left( \frac{n}{n-k'}\right)^{\frac{n-k'}{2k'}} |s_{k'}|^{\frac{1}{k'}}
\end{align*}
and hence on substituting \eqref{s-major} and $|s_n|=1$ we conclude that
$$1 \ll A^{\frac{n-2k'}{k'}} \left( \frac{n}{n-k'}\right)^{\frac{n-k'}{2k'}}
\left( \frac{k'}{n}\right)^{\frac{1}{2}}
$$
for some $k'=k,k+1$.  If $k \geq 2n/3$, then the terms $\left( \frac{n}{n-k'}\right)^{\frac{n-k'}{2k'}}$ and $\left( \frac{k'}{n}\right)^{\frac{1}{2}}$ are bounded and $\frac{n-2k'}{k'} \leq -\frac{1}{2}$, leading to a contradiction since $A$ is large.  Thus we may assume\footnote{In fact we could improve this to $k \leq n/2$ with a little more effort, but for our arguments the only thing that is important is that we can ensure that $n-k' \asymp n-k \asymp n$ for $k'=k,k+1$.} that $k < 2n/3$.

From Theorem \ref{new-ineq} and our choice of normalizations, we have the inequality
\begin{equation}\label{contra}
    \sum_{m=0}^n \binom{n}{m} |s_m| r^m \geq (1+r^2)^{n/2}
\end{equation}
for any $r>0$.  We will show that for $A$ large enough, this bound is inconsistent with \eqref{s-major}, \eqref{A-def}, \eqref{flip} for a suitable choice of $r$.  In order to do this, we must first obtain upper bounds for $\binom{n}{m} |s_m|$ for all $0 \leq m \leq n$ (not just $m=k,k+1$).

In the region $k \leq m \leq n$, we may simply apply \eqref{A-def} and \eqref{s-major} to obtain the upper bound\footnote{We adopt the convention that the asymptotic $O()$ notation is applied before exponentiation if the exponent appears outside the parentheses; thus for instance $O(\frac{m}{n})^{m/2}$ denotes a quantity bounded by $(C \frac{m}{n})^{m/2}$ for some absolute constant $C>0$.}
$$
|s_m| \leq O\left(\frac{m}{n}\right)^{m/2}
$$
and hence by \eqref{standard}
\begin{equation}\label{mbound}
   \binom{n}{m} |s_m| \leq O\left(\frac{n}{m}\right)^{m/2}.
\end{equation}

Now we consider the problem of bounding $s_m$ in the complementary range $0 \leq m < k$.  To do this, we apply \eqref{flip} to the reversed sequence $\pm (s_n,\dots,s_0)$ (which is attainable for some choice of sign $\pm$ thanks to Lemma \ref{syms}(ii) and the normalization $|s_n|=1$) and with $k, \ell$ replaced by $k-m, n-m$ respectively to conclude that
\begin{align*}
     |s_m|^{\frac{1}{n-m}}
     &= \max_{k' = k-m, k-m+1} A^{\frac{k'}{n-m-k'}} \left( \frac{n-m}{k'}\right)^{\frac{k'}{2(n-m-k')}} |s_{k'+m}|^{\frac{1}{n-m-k'}} \\
     &= \max_{k' = k, k+1} A^{\frac{k'-m}{n-k'}} \left( \frac{n-m}{k'-m}\right)^{\frac{k'-m}{2(n-k')}} |s_{k'}|^{\frac{1}{n-k'}}
\end{align*}
and thus by \eqref{s-major} (and noting that $\frac{(n-m)k'}{n-k}= O(k') = O(N)$ since $k \leq 2n/3$ and $n \leq N$ is large)
$$ |s_m| \ll A^{\frac{(k'-m) (n-m)}{n-k'}} \left( \frac{n-m}{k'-m}\right)^{\frac{(k'-m)(n-m)}{2(n-k')}} \left( A^{-1} \left(\frac{k'}{n}\right)^{1/2} \right)^{\frac{(n-m)k'}{n-k'}}$$
for some $k'=k,k+1$.
This bound is somewhat complicated, so we shall now simplify the right-hand side (conceding some factors roughly of the form $O(1)^m$, which will turn out to be acceptable losses).
From the trivial bound $\frac{n-m}{k'-m} \leq \frac{n}{k'-m}$, we conclude that
$$ |s_m| \ll A^{-\frac{m (n-m)}{n-k'}} \left( \frac{n}{k'-m}\right)^{\frac{(k'-m)(n-m)}{2(n-k')}} \left(\frac{k'}{n}\right)^{\frac{(n-m)k'}{2(n-k')}}.$$
Since $\frac{n-m}{n-k'} \geq 1$, we may replace the $A^{-\frac{m (n-m)}{n-k'}} $ term with $A^{-m}$.
Next, we use the Bernoulli inequality bound
$$ \left(\frac{k'}{k'-m}\right)^{k'-m} \leq \exp\left( \frac{m}{k'-m} (k'-m) \right) = e^m$$
to conclude that
$$ |s_m| \ll O(A^{-1})^m \left( \frac{n}{k'}\right)^{-\frac{m(n-m)}{2(n-k')}}.$$
By \eqref{standard}, we thus have
\begin{align*}
    \binom{n}{m} |s_m| &\ll O(A^{-1})^m \left( \frac{n}{m} \right)^m \left( \frac{n}{k'}\right)^{-\frac{m(n-m)}{2(n-k')}} \\
    &= O\left( \frac{k'}{Am} \right)^m  \left( \frac{n}{k'}\right)^{\frac{m}{2} - \frac{m(k'-m)}{2(n-k')}}.
\end{align*}
Since $\frac{m(k'-m)}{n-k'} \geq 0$, $\frac{n}{k'} \geq 1$, and $k' \asymp k$, we conclude that
\begin{equation}\label{m-1}
    \binom{n}{m} |s_m| \leq O\left( \frac{k}{Am} \right)^m \left( \frac{n}{k}\right)^{\frac{m}{2}}
\end{equation}
for $m \geq 1$.  For $m=0$ the left-hand side is of course equal to $1$.

We now insert the bounds \eqref{mbound}, \eqref{m-1} into \eqref{contra} with
$$ r \coloneqq \delta \left( \frac{k}{n}\right)^{1/2}$$
and $\delta>0$ a small absolute constant to be chosen later.  We conclude that
\begin{align*}
    (1 + \delta^2 k/n)^{n/2} &\leq 1 + \sum_{1 \leq m < k} O\left( \frac{k \delta}{Am} \right)^m +
    \sum_{m=k}^\infty O\left(\frac{\delta k^{1/2}}{m^{1/2}}\right)^m \\
     &\leq 1 + \sum_{m=1}^\infty \frac{O(k\delta/A)^m}{m!} + \sum_{m=k}^\infty O(\delta)^m \\
     &\leq \exp( O(\delta k/A) ) + 2^{-k}
\end{align*}
(say) if $\delta$ is small enough; on the other hand, we have
$$ 1+\delta^2 k/n \geq \exp( \delta^2 k/2n )$$
if $\delta$ is small enough.  We conclude that
$$ \exp( O(\delta k/A) ) + 2^{-k} \geq \exp(\delta^2 k / 4).$$
If $A$ is sufficiently large depending on $\delta$, this implies
$$ \exp( \delta^2 k / 8 ) + 2^{-k} \geq \exp(\delta^2 k / 4),$$
but this leads to a contradiction if $k$ is large enough.  This completes the proof of Theorem \ref{thm-main}.

\end{document}